\documentclass[12pt]{amsart}
\usepackage[T1]{fontenc}
\usepackage[utf8]{inputenc}

\usepackage[mathcal]{euler}   

\usepackage{amsmath,amssymb,amsthm}
\usepackage{mathtools}        
\allowdisplaybreaks

\usepackage[numbers,sort&compress]{natbib}

\usepackage[left=1.25in, bottom=1in, right=1in, top=1in,
            paperwidth=8.5in, paperheight=11in]{geometry}
\usepackage{setspace} 
\usepackage{placeins} 

\usepackage[shortlabels]{enumitem}

\usepackage{array}
\usepackage{longtable}
\usepackage{afterpage}
\usepackage{float}
\usepackage{pdflscape}

\usepackage{graphicx}
\usepackage{epsfig}
\usepackage{epstopdf}
\usepackage{psfrag}

\usepackage{tikz}
\usetikzlibrary{calc,shapes.geometric}
\usepackage{tikz-cd}
\usepackage{ytableau}
\usepackage{varwidth}

\usepackage{xy}
\xyoption{all}

\usepackage{slashed}
\usepackage{mathdots}   

\usepackage{marginnote}
\usepackage[ddmmyyyy]{datetime}
\usepackage[normalem]{ulem}   

\usepackage{xcolor}

\usepackage[unicode=true,
  linktocpage,
  linkbordercolor={0.5 0.5 1},
  citebordercolor={0.5 1 0.5},
  linkcolor=blue]{hyperref}

\makeatletter
\newcommand*\bigcdot{\mathpalette\bigcdot@{.5}}
\newcommand*\bigcdot@[2]{\mathbin{\vcenter{\hbox{\scalebox{#2}{$\m@th#1\bullet$}}}}}
\makeatother

\newcommand{\Q}{\mathbb{Q}}

\newcommand{\andd}{\qquad\text{ and }\qquad}

\newcommand{\orr}{\qquad\text{ or }\qquad}

\newcommand{\steph}[1]{\textcolor{black}{#1}}

\swapnumbers
\numberwithin{equation}{section}

\theoremstyle{plain}
\newtheorem*{theorem*}{Theorem}
\newtheorem*{lemma*}{Lemma}
\newtheorem{theorem}[equation]{Theorem}
\newtheorem{lemma}[equation]{Lemma}
\newtheorem{proposition}[equation]{Proposition}
\newtheorem{corollary}[equation]{Corollary}

\theoremstyle{definition}
\newtheorem{definition}[equation]{Definition}
\newtheorem*{definition*}{Definition}
\newtheorem{example}[equation]{Example}

\newtheorem*{notation*}{Notation}

\theoremstyle{remark}

\newtheorem*{remark*}{Remark}

\begin{document}

\title[Schur-positive multipartite graphs]{A Schur-positivity classification for
complete multipartite graphs}

\author{Ethan Shelburne}
\address{
Department of Computer Science,
 University of British Columbia,
 Vancouver BC V6T 1Z4, Canada}
\email{shelburneethan@gmail.com}
\author{Stephanie van Willigenburg}
\address{
 Department of Mathematics,
 University of British Columbia,
 Vancouver BC V6T 1Z2, Canada}
\email{steph@math.ubc.ca}

\thanks{{Both} authors were supported in part by the Natural Sciences and Engineering Research Council of Canada.}

\begin{abstract}
A graph is Schur-positive if its chromatic symmetric function expands non-negatively in the Schur basis. We determine a full Schur-positivity classification for complete multipartite graphs by showing that a complete multipartite graph $K_\lambda$ is Schur-positive if and only if either $\lambda_i\in \{1,2\}$ for all $i$ or $\lambda=(3,2^\beta)$ for some $\beta\ge 1$. These results extend earlier classifications for complete bipartite and complete tripartite graphs to full generality. Our proofs combine structural arguments ruling out most cases, with a combinatorial analysis of Schur coefficients for the remaining family $K_{(3,2^\beta)}$ via special rim hook $G$-tabloids. Along the way, we establish a simpler formula for Schur coefficients of incomparability graphs, which we then apply to compute the coefficients of interest in terms of non-increasing sequences.
\end{abstract}

\maketitle

\section{Introduction}

In 1995, Richard Stanley extended the study of graph colorings to the realm of symmetric functions by introducing the chromatic symmetric function \cite{stan_95}. This innovation sparked extensive research into chromatic symmetric functions and related power series, uncovering deep connections to representation theory and algebraic structures. Symmetric functions that expand non-negatively in the elementary basis or Schur basis, termed  $\mathbf{e}$-positive or Schur-positive respectively, often possess significant algebraic or combinatorial properties, including meaningful interpretations of their coefficients. For instance, symmetric functions with non-negative Schur basis expansions correspond to the Frobenius image of representations of the symmetric group \cite{sym_group}. A central question posed in \cite{stan_95} seeks to classify graphs that are Schur-positive or $\mathbf{e}$-positive, leading to three major problems that continue to shape the field, which we frame as conjectures.
\begin{enumerate}

\item (The Stanley-Stembridge Conjecture, \cite{stan_93}): All claw-free incomparability graphs are $\mathbf{e}$-positive.

\item (The Nonisomorphic Tree Conjecture, \cite{stan_95}): No two non-isomorphic trees have the same chromatic symmetric function.

\item (The Claw-free Conjecture, \cite{stan_98}): All claw-free graphs are Schur-positive.

\end{enumerate}

A proof of the Stanley-Stembridge Conjecture was given in \cite{hikita24}. Additionally, various methods have been employed to show subfamilies of claw-free incomparability graphs satisfy this conjecture in \cite{centered_at_vertex, Dahl18, GebSag01}, among others. The Nonisomorphic Tree Conjecture has been computationally verified for trees with up to 29 vertices in \cite{trees29}, and further progress on this conjecture is detailed in \cite{AdMOZ}, for instance. Vesselin Gasharov advanced the Claw-free Conjecture by demonstrating that all claw-free incomparability graphs are Schur-positive in \cite{gash_96}. A key question towards this from the opposite direction is to classify when complete multipartite graphs are Schur-positive. This is because complete multipartite graphs are incomparability graphs that almost always contain the claw. David Wang and Monica Wang provided a combinatorial formula for determining the Schur coefficients of chromatic symmetric functions, which facilitated various results, including a classification of the Schur-positivity of complete bipartite and tripartite graphs \cite{WW2020}. Building on the combinatorial formula of \cite{WW2020}, special rim hook G-tabloids were introduced as a new framework for expressing Schur coefficients \cite{gen_nets}. This perspective not only provides a more tractable combinatorial interpretation, but also enabled the establishment of Schur-positivity for all generalized net graphs \cite{gen_nets}. These developments highlight the potential of such combinatorial models as a systematic tool for studying Schur-positivity, motivating the further investigation undertaken in this work.



In this paper, we fully answer the aforementioned key question, and our paper is structured as follows. Section \ref{sect:prelims} covers the required preliminaries. Section \ref{sect:srh_G_tabs} introduces the key combinatorial object, that is, the special rim hook $G$-tabloid. Lastly, Section \ref{sect:comp_multi_graphs} includes several results which together provide a classification for the Schur-positivity of all complete multipartite graphs in Theorem~\ref{thm:classif}.

\section{Preliminaries}
\label{sect:prelims}
 A \textit{partition} of $n\ge 0$ is a sequence of weakly decreasing positive integers $\lambda=(\lambda_1,\ldots,\lambda_k)$ whose sum is $n$.
 The \textit{length} of $\lambda$ is given by $\ell(\lambda)=k$. In the case where $n=0$, we say $\lambda$ is the \textit{empty partition}. We use exponents to indicate repeated integers in a partition. For instance, $\lambda=(5,5,3,3,1)=(5^2,3^2,1)$. The \textit{diagram} of a partition $\lambda$ of $n$ is an array of $n$ boxes (called \textit{cells}) in left-justified rows such that row $i$ contains $\lambda_i$ boxes, where the rows are indexed from top to bottom and the columns are indexed from left to right. Below, we depict the diagram of the partition $(5,5,3,3,1)$.


 \begin{figure}[H]
 \begin{tikzpicture}
 \node[scale=.6] at (0,0) (a) {\ydiagram{5,5,3,3,1}};
  \end{tikzpicture}
\label{fig:young_diagram_example}
 \end{figure}

 A \textit{composition} of $n\ge 0$ is a sequence of positive integers $\kappa=[\kappa_1,\ldots,\kappa_j]$ whose sum is $n$. The \textit{length} of $\kappa$ is given by $\ell(\kappa)=j$. In the case where $n=0$, we say $\kappa$ is the \textit{empty composition}.
 Again, we may use exponents to denote repeated integers in a composition. Given a composition $\kappa$, we use $\Lambda(\kappa)$ to denote the partition obtained by arranging the integers in $\kappa$ in weakly decreasing order. 


A \textit{partially ordered set} or \textit{poset} is a set $P$ equipped with a binary relation $\le$ that is reflexive, antisymmetric, and transitive. We say that $x,y\in P$ are \textit{comparable} if $x\le y$ or $y\le x$, and we say they are \textit{incomparable} otherwise. A \textit{total order} is a partial order where every pair of elements is comparable. We write $x<y$ if $x\le y$ and $x\ne y$. We say $x$ \textit{covers} $y$ and write $y\lessdot x$ if $y<x$ and there is no $z\in P$ such that $y<z<x$. The \textit{Hasse diagram} of $P$ is the graph with vertices $P$ and an edge from $y$ up to $x$ if $y\lessdot x$.

\begin{example}
\label{ex:hasse}
Consider the poset $P=\{a,b,c,d,e,f\}$ with relation $\le$ satisfying precisely the conditions
\[
a\lessdot b\lessdot f,\qquad a\lessdot c \lessdot e,\qquad d\lessdot c,\andd b\lessdot e.
\]
The Hasse diagram of $P$ is depicted below.
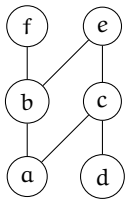
\begin{figure}[H]
    \centering
         \begin{tikzpicture}
\node[circle,draw,scale=.75] at (0,0) (a) {$a$};
\node[circle,draw,scale=.75] at (0,1) (b) {$b$};
\node[circle,draw,scale=.75] at (0,2) (f) {$f$};
\node[circle,draw,scale=.75] at (1,0) (d) {$d$};
\node[circle,draw,scale=.75] at (1,1) (c) {$c$};
\node[circle,draw,scale=.75] at (1,2) (e) {$e$};
\draw (a)--(b)--(f);
\draw (d)--(c)--(e);
\draw (a)--(c);
\draw (b)--(e);
\end{tikzpicture}
\end{figure}
\end{example}

    The \textit{incomparability graph} of a poset $P$ is the graph $\text{inc}(P)$ on vertices corresponding to the elements of $P$ such that, for all $x,y\in P$, $x$ is adjacent to $y$ if and only if $x$ and $y$ are incomparable. It is important to note that not every graph is an incomparability graph of some poset.

\begin{example}
\label{ex:inc_graph}
The incomparability graph $\text{inc}(P)$ of the poset $P$ in Example \ref{ex:hasse} is depicted below.
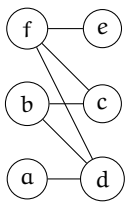
\begin{figure}[H]
    \centering
         \begin{tikzpicture}
\node[circle,draw,scale=.75] at (0,0) (a) {$a$};
\node[circle,draw,scale=.75] at (0,1) (b) {$b$};
\node[circle,draw,scale=.75] at (0,2) (f) {$f$};
\node[circle,draw,scale=.75] at (1,0) (d) {$d$};
\node[circle,draw,scale=.75] at (1,1) (c) {$c$};
\node[circle,draw,scale=.75] at (1,2) (e) {$e$};
\draw (a)--(d)--(f);
\draw (d)--(b)--(c);
\draw (f)--(c);
\draw (f)--(e);
\end{tikzpicture}
\end{figure}
\end{example}

Next, let $\mathbf{x} = {x_1, x_2, x_3, \ldots}$ be a countably infinite collection of commuting variables. We consider the \textit{algebra of formal power series} over the rational numbers in the variables $\mathbf{x}$, denoted by $\Q[[\mathbf{x}]]$. A function $f(\mathbf{x})$ is termed \emph{symmetric} if it is invariant under any permutation of the variables $\mathbf{x}$. The subspace of symmetric functions
\[
\text{Sym}(\mathbf{x})=\Q [e_1, e_1, e_3, \ldots]
\]
where $e_1, e_1, e_3, \ldots$ are defined below, forms another algebra known as the \textit{algebra of symmetric functions}.

We first focus on the classical Schur basis for Sym$(\mathbf{x})$, which we introduce using semistandard Young tableaux. A \textit{semistandard Young tableau} (SSYT) of \textit{shape} $\lambda$ is a filling $Q$ of the cells of the diagram $\lambda$ with positive integers such that rows weakly increase from left to right and columns strictly increase from top to bottom.  Given a semistandard Young tableau $Q$, we define the \textit{weight} of $Q$ to be
    \[
    \text{wt}(Q)=x_1^{\#1s}x_2^{\#2s}x_3^{\#3s}\cdots.
    \]

\begin{example}
\label{ex:SSYT}
 We portray several examples of SSYTs of shape (4,2,1) below. 
 
 
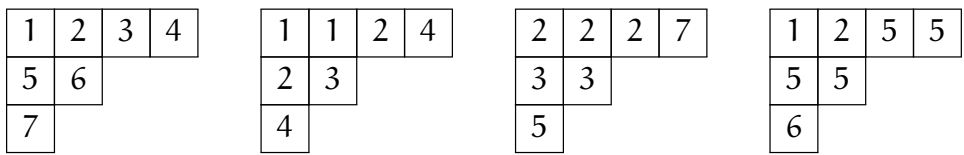
\begin{figure}[H]
\begin{align*}
\begin{ytableau}
       1 &  2 & 3  & 4 \\
       5  & 6  \\
        7
\end{ytableau}
\qquad
\begin{ytableau}
       1 &  1 &  2 & 4 \\
       2  & 3  \\
        4
\end{ytableau}
\qquad
\begin{ytableau}
       2 &  2 & 2  & 7 \\
       3  & 3  \\
        5
\end{ytableau}
\qquad
\begin{ytableau}
       1 &  2 & 5  & 5 \\
       5  & 5  \\
        6
\end{ytableau}
\end{align*}

\label{fig:SSYTs}
\end{figure}
 
 From left to right, these SSYTs have weights
 \begin{align*}
 \label{eq:SSYTweights}
 x_1x_2x_3x_4x_5x_6x_7,\qquad x_1^2x_2^2x_3x_4^2,\qquad x_2^3x_3^2x_5x_7,\andd x_1x_2x_5^4x_6.
 \end{align*}

\end{example}

Given some partition $\lambda$, the \textit{Schur function} associated to $\lambda$ is
\[
s_{\lambda}=\sum_{Q}\text{wt}(Q)
\]
where the sum spans over all semistandard Young tableaux $Q$ of shape $\lambda$.

Next, we introduce the classical elementary basis for Sym$(\mathbf{x})$. We define the \textit{$i$th elementary symmetric function} to be
    \[
    e_i=s_{(1^i)}
    \]
    and the \textit{elementary symmetric function} associated to a partition $\lambda$ to be
    \[
    e_{\lambda}=e_{\lambda_1}\cdots e_{\lambda_{k}}.
    \]

We have that
\[
\mathbf{s}=\{s_{\lambda}\,\mid\, \lambda\text{ is a partition}\}\andd \mathbf{e}=\{e_{\lambda}\,\mid\, \lambda\text{ is a partition}\}
\]
are each bases for $\text{Sym}(\mathbf{x})$. If $\textbf{b}$ is a basis for $\text{Sym}(\mathbf{x})$ which is indexed by partitions, we use the notation
\[
[b_{\lambda}]f(\mathbf{x})= \text{the coefficient of $b_{\lambda}$ in the expansion of $f(\mathbf{x})\in \text{Sym}(\mathbf{x})$ in the basis $\mathbf{b}$}. 
\]
Given some $f(\mathbf{x})\in \text{Sym}(\mathbf{x})$, we say $f(\mathbf{x})$ is \textit{Schur-positive} if $[s_{\lambda}]f(\mathbf{x})\ge 0$ for all partitions $\lambda$. We say $f(\mathbf{x})$ is \textit{$\mathbf{e}$-positive} if $[e_{\lambda}]f(\mathbf{x})\ge 0$ for all partitions $\lambda$. One well-known important fact regarding these two bases is that the elementary symmetric functions themselves are Schur-positive. Accordingly, the $\mathbf{e}$-positivity of $f(\mathbf{x})$ implies the Schur-positivity of $f(\mathbf{x})$. Likewise, if $f(\mathbf{x})$ is not Schur-positive, then $f(\mathbf{x})$ is also not $\mathbf{e}$-positive.

Next, we define a family of functions belonging to Sym$(\mathbf{x})$ known as chromatic symmetric functions, which were introduced in \cite{stan_95}. Given a graph $G$ with vertices $V(G) = \{v_1, v_2, \ldots , v_n\}$, a \textit{proper coloring} of $G$ in $q$ colors is a map
\[
\mathcal{C}:V(G)\to\{1,2,3,\ldots, q\}
\]
such that, if $u$ and $v$ are adjacent, then $\mathcal{C}(u)\ne \mathcal{C}(v)$. For a proper coloring $\mathcal{C}$ of $G$, we define
    \[
\mathbf{x}^{\mathcal{C}}=x_{\mathcal{C}(v_1)}x_{\mathcal{C}(v_2)}\cdots x_{\mathcal{C}(v_n)}.
    \]
    The \textit{chromatic symmetric function} of a graph $G$ is the formal power series
\[
X_G(\mathbf{x})=\sum_{\mathcal{C}}\mathbf{x}^{\mathcal{C}},
\]
where the sum ranges over all proper colorings $\mathcal{C}$ of $G$.
    The chromatic symmetric function is symmetric because permuting the variables of the function is equivalent to permuting the colors in each proper coloring. We say that a graph $G$ is \textit{Schur-positive} (resp. \textit{$\mathbf{e}$-positive}) if $X_G(\mathbf{x})$ is Schur-positive (resp. $\mathbf{e}$-positive).

There are certain structural properties associated with graphs which are either Schur-positive or $\mathbf{e}$-positive. In order to discuss a useful example of such a property, we require several more definitions. First, we consider the \textit{dominance partial order} on partitions. For partitions $\lambda=(\lambda_1,\ldots,\lambda_k)$ and $\mu=(\mu_1,\ldots,\mu_j)$, we say $\mu$ is \textit{dominated} by $\lambda$ and write $\lambda \ge \mu$ if
\[\lambda_1+\cdots+\lambda_i \ge \mu_1+\cdots+\mu_i
\]
for all $i\ge 1$. For this definition, we set $\lambda_i=0$ for $i>k$ and $\mu_i=0$ for $i>j$.

Next, for a graph $G$ with vertices $V(G)$, we define a \textit{stable set} in $G$ to be a subset $U \subseteq V(G)$ for which all vertices in $U$ are pairwise non-adjacent. If $|U|=m$, we may refer to $U$ as a \textit{stable $m$-set}. For a partition $\lambda=(\lambda_1,\ldots,\lambda_k)$, we define a \textit{stable partition of $G$ of type $\lambda$} to be a division of all the vertices $V(G)$ into $k$ disjoint stable sets whose cardinalities are given by $\lambda_1,\ldots,\lambda_k$. We refer to each stable set of cardinality $\lambda_i$ as a \emph{$\lambda_i$-part.} We may also refer to a stable partition of $G$ of type $\lambda$ as a \textit{stable $k$-partition of $G$} when $\ell(\lambda)=k$. Additionally, we call a stable partition of $G$ \textit{semi-ordered} if the stable sets of the same cardinality are assigned an order.

Building on the prior definitions, we say a graph $G$ is \textit{nice} if whenever $G$ has a stable partition of type $\lambda$ and $\lambda\ge \mu$, then $G$ also has a stable partition of type $\mu$. The contrapositive of the following proposition (\cite[Proposition 1.5]{stan_98}) is a useful tool for disproving the Schur-positivity of certain graphs.

\begin{proposition}[\cite{stan_98}]
\label{prop:schur_pos_means_nice}
If $G$ is Schur-positive, then $G$ is nice.
\end{proposition}

With the objective of discussing an important result from \cite{WW2020}, we now introduce a combinatorial object known as a special rim hook tabloid.

 A \textit{rim hook} of \textit{length} $k$ is a sequence of $k$ connected cells in a diagram, each of which lies on the southeast boundary, and whose removal results in a smaller diagram. For any rim hook, we call each path between consecutive cells a \textit{step}. If the cells are in different rows, we use the term \textit{north step} ($N$-step), whereas if the cells are in different columns, we use the term $\textit{east step}$ ($E$-step). We sometimes refer to a rim hook by indicating the sequence of steps it includes, read from bottom to top and left to right. For example, an $NEN$-hook is a rim hook spanning four cells and consisting of an $N$-step then an $E$-step followed by another $N$-step.

 Let $\kappa=[\kappa_1,\ldots,\kappa_{j}]$ be a composition and $\lambda=(\lambda_1,\ldots,\lambda_k)$ be a partition. A \textit{rim hook tabloid of shape $\lambda$ and content} $\kappa$ is a filling of the cells of the diagram of $\lambda$ with $j$ sequences of connected cells $r_i$ such that $r_1$ is a rim hook of length $\kappa_1$ and, for all $1\le i \le j-1$, if $r_1,\ldots, r_i$ are removed from $\lambda$ to form $\tilde{\lambda}$, $r_{i+1}$ is a rim hook of $\tilde{\lambda}$ of length $\kappa_{i+1}$.  A \textit{special rim hook tabloid} (SRH tabloid) is a rim hook tabloid such that every rim hook intersects the first column. We define the \textit{sign} of an SRH tabloid $T$ to be
 \[
 \mathrm{sgn}(T)=(-1)^{\text{\# north steps in $T$}}.
 \]
 Moreover, we use the notation $\mathcal{T}_{\lambda}$ to denote the set of all SRH tabloids of shape $\lambda$. Given an SRH tabloid $T$, we denote by $\kappa_T$ the \textit{content} of $T$, which is the composition given by the rim hook lengths read from the bottom rim hook to the top rim hook of the diagram.

 \begin{example}
We portray below all the possible SRH tabloids of shape $\lambda=(4,2,2)$, that is, all elements of the set $\mathcal{T}_{(4,2,2)}$.

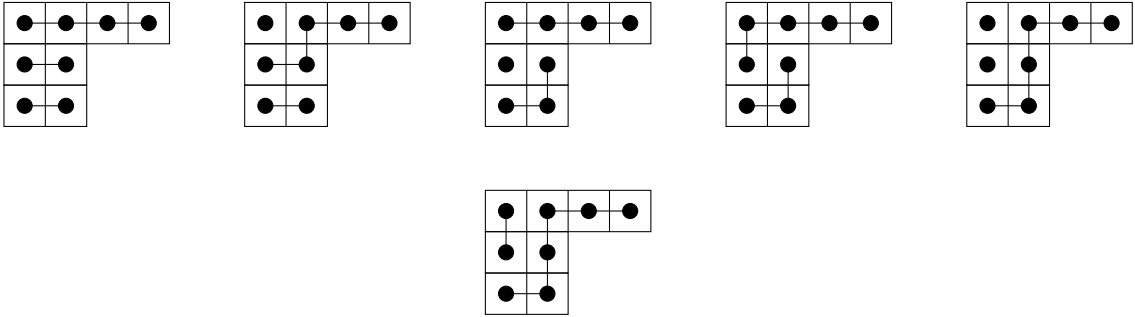
\begin{figure}[H]
\centering
\begin{tikzpicture}[
    dot/.style={circle, fill=black, inner sep=1pt},
    cellcircle/.style={draw, fill=black, circle, inner sep=0pt, minimum size=0.2cm},
    line width=0.4pt
]

\def\cell{0.55}   

\newcommand{\srhshape}{
    \draw (0,0) rectangle (4*\cell,-\cell);
    \draw (0,-\cell) rectangle (2*\cell,-2*\cell);
    \draw (0,-2*\cell) rectangle (2*\cell,-3*\cell);

    \foreach \x in {1,2,3}
        \draw (\x*\cell,0) -- (\x*\cell,-\cell);

    \draw (\cell,-\cell) -- (\cell,-3*\cell);

    \foreach \x in {0,...,3}
        \node[cellcircle] at ({(\x+0.5)*\cell},{-0.5*\cell}) {};

    \foreach \x in {0,1}
        \node[cellcircle] at ({(\x+0.5)*\cell},{-1.5*\cell}) {};

    \foreach \x in {0,1}
        \node[cellcircle] at ({(\x+0.5)*\cell},{-2.5*\cell}) {};
}

\newcommand{\srhtabloid}[3]{%
\begin{scope}[shift={(#1,#2)}]
    \srhshape

    \coordinate (11) at (0.5*\cell,-0.5*\cell);
    \coordinate (12) at (1.5*\cell,-0.5*\cell);
    \coordinate (13) at (2.5*\cell,-0.5*\cell);
    \coordinate (14) at (3.5*\cell,-0.5*\cell);
    \coordinate (21) at (0.5*\cell,-1.5*\cell);
    \coordinate (22) at (1.5*\cell,-1.5*\cell);
    \coordinate (31) at (0.5*\cell,-2.5*\cell);
    \coordinate (32) at (1.5*\cell,-2.5*\cell);

    \node[dot] at (11) {};
    \node[dot] at (12) {};
    \node[dot] at (13) {};
    \node[dot] at (14) {};
    \node[dot] at (21) {};
    \node[dot] at (22) {};
    \node[dot] at (31) {};
    \node[dot] at (32) {};

    #3
\end{scope}
}

\def\dx{3.2}

\srhtabloid{0}{0}{
    \draw (11)--(12)--(13)--(14);
    \draw (21)--(22);
    \draw (31)--(32);
}

\srhtabloid{\dx}{0}{
    \draw (21)--(22)--(12)--(13)--(14);
    \draw (31)--(32);
}

\srhtabloid{2*\dx}{0}{
    \draw (11)--(12)--(13)--(14);
    \draw (31)--(32)--(22);
}

\srhtabloid{3*\dx}{0}{
    \draw (21)--(11)--(12)--(13)--(14);
    \draw (31)--(32)--(22);
}

\srhtabloid{4*\dx}{0}{
    \draw (31)--(32)--(22)--(12)--(13)--(14);
}

\srhtabloid{2*\dx}{-2.5}{
    \draw (11)--(21);
    \draw (31)--(32)--(22)--(12)--(13)--(14);
}

\end{tikzpicture}
\end{figure}

Counting the parity of the number of $N$-steps, we find the second, third, and sixth tabloids have negative sign whereas the rest of the tabloids have positive sign. To further illustrate notation, we also remark that the fourth tabloid consists of an $EN$-hook and an $NEEE$-hook.

 \end{example}

  Considering SRH tabloids and semi-ordered stable partitions allows us to state the following result from \cite[Theorem 3.1]{WW2020}.
 
 \begin{theorem}\cite{WW2020}
 \label{thm:WW}
 For any graph $G$ and any partition $\lambda$ of $|V(G)|$, we have
 \begin{align*}
     [s_{\lambda}]X_G=\sum_{T\in \mathcal{T}_{\lambda}}\mathrm{sgn}(T)so_G(T),
 \end{align*}
 where $so_G(T)$ denotes the number of semi-ordered stable partitions of $G$ of type $\Lambda(\kappa_T)$.
 \end{theorem}

 In \cite{WW2020}, the formula is utilized to prove the non-Schur-positivity of subclasses of complete tripartite graphs, squid graphs, and pineapple graphs. In this work, we employ an alternative version of the above formula (Proposition \ref{prop:better_s_form}), which was proved in \cite[Corollary 3.1]{gen_nets}. 

Lastly, we define a graph $G$ to be \textit{claw-free} if $G$ has no induced subgraphs isomorphic to the claw graph $K_{(3,1)}$, which we depict below.

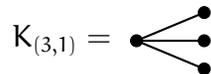
\begin{figure}[H]
\centering
\begin{tikzpicture}
\node[] at (-1,.75) {$K_{(3,1)}=$};

\draw[fill=black,scale=.75] (0,1) circle (3pt);
\draw[fill=black,scale=.75] (1.2,1.5) circle (3pt);
\draw[fill=black,scale=.75] (1.2,1) circle (3pt);
\draw[fill=black,scale=.75] (1.2,.5) circle (3pt);

\draw[thick,scale=.75] (0,1) -- (1.2,1.5);
\draw[thick,scale=.75] (0,1) -- (1.2,1);
\draw[thick,scale=.75] (0,1) -- (1.2,.5);
\end{tikzpicture}
\end{figure}

Much of the work surrounding Schur-positivity is motivated by the open conjecture that all claw-free graphs are Schur-positive \cite{stan_98}. It was shown in \cite{gash_96} that all claw-free incomparability graphs satisfy this conjecture. The stronger statement that all claw-free incomparability graphs are $\mathbf{e}$-positive, i.e. the Stanley-Stembridge Conjecture, was proved in \cite{hikita24}.

\section{Special Rim Hook $G$-Tabloids}
\label{sect:srh_G_tabs}

  In this section, we define an SRH $G$-tabloid, which is a modified version of an SRH tabloid that corresponds to a graph $G$. We then state a formula for all Schur coefficients of chromatic symmetric functions in terms of signed sums of these combinatorial objects, which was proved in \cite[Corollary 3.1]{gen_nets}.

 \begin{definition}
 \label{def:SRH_G_tabs}
 Consider a graph $G$ and a partial order $\le$ on the vertices of $G$ satisfying: if vertices $u$ and $v$ are non-adjacent, then $u$ and $v$ are comparable. We can then define a \textit{special rim hook $G$-tabloid} (or SRH $G$-tabloid) to be an SRH tabloid such that every cell is filled with a vertex of $G$ and the following conditions are met.
\begin{itemize}
    \item Cells spanned by the same rim hook contain vertices which form a stable set.
    \item For each rim hook, reading the corresponding vertices from southwest to northeast order results in an increasing sequence with respect to the partial order.
\end{itemize}
The \textit{sign} and \textit{shape} of an SRH $G$-tabloid are respectively the sign and shape of the underlying SRH tabloid.
We denote the set of all SRH $G$-tabloids of shape $\lambda$ by $\mathcal{T}_{\lambda,G}$.
 \end{definition}


\begin{definition}
\label{def:tail_head}
Let $T$ be an SRH $G$-tabloid. We define the \textit{tail} of $T$, denoted \text{tl}$(T)$, to be the part of $T$ containing all rows of length $1$. Likewise, we define the \textit{head} of $T$, denoted $\text{hd}(T)$, to be the part of $T$ containing all rows of length strictly greater than $1$. Either tl$(T)$ or hd$(T)$ could be empty, depending on the shape of $T$. 
\end{definition}

Let $T$ and $T'$ be SRH $G$-tabloids. We say that
\[
\text{tl}(T)=\text{tl}(T') \orr \text{hd}(T)=\text{hd}(T')
\]
if the tails (respectively, heads) of $T$ and $T'$ are equal as SRH $G'$-tabloids, where $G'$ is the induced subgraph of $G$ on the set of vertices appearing in the tails (respectively, heads).


 The choice of a partial order introduces flexibility when working with SRH $G$-tabloids. When $G$ is the incomparability graph of a poset $P$, a natural approach is to adopt the partial order defined by the poset itself. In an incomparability graph, two vertices are non-adjacent precisely when they are comparable. This ensures that the necessary condition is automatically met. Otherwise, if $G$ is not an incomparability graph, the simplest approach is to choose a total order on $G$ by labeling the vertices numerically.

\begin{example}
\label{ex:SRH_G_tab}

Consider the poset $P$ depicted in Example \ref{ex:hasse} and the corresponding incomparability graph $G=\text{inc}(P)$ depicted in Example \ref{ex:inc_graph}. We depict below some SRH $G$-tabloids of shape $(2,1^4)$, that is, some elements of the set $\mathcal{T}_{(2,1^4),G}$. The heads of the leftmost and rightmost are equal as SRH $G'$-tabloids for the subgraph $G'$ of $G$ consisting of the vertices $a$ and $c$.


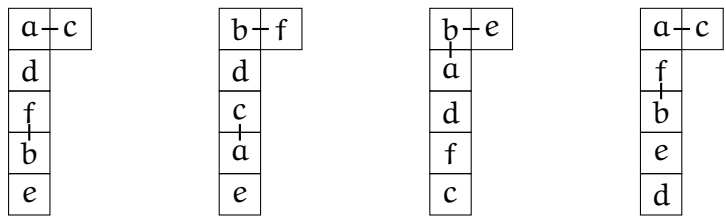
\begin{figure}[H]
\centering
\begin{tikzpicture}[
    line width=0.4pt
]

\def\cell{0.55}

\newcommand{\gtabloidshape}{
    \draw (0,0) rectangle (2*\cell,-\cell);
    \draw (\cell,0) -- (\cell,-\cell);

    \draw (0,-\cell) rectangle (\cell,-2*\cell);
    \draw (0,-2*\cell) rectangle (\cell,-3*\cell);
    \draw (0,-3*\cell) rectangle (\cell,-4*\cell);
    \draw (0,-4*\cell) rectangle (\cell,-5*\cell);
}

\newcommand{\gtabloid}[9]{%
\begin{scope}[shift={(#1,#2)}]
    \gtabloidshape

    \node at (0.5*\cell,-0.5*\cell) {$#3$};   
    \node at (1.5*\cell,-0.5*\cell) {$#4$};   
    \node at (0.5*\cell,-1.5*\cell) {$#5$};   
    \node at (0.5*\cell,-2.5*\cell) {$#6$};   
    \node at (0.5*\cell,-3.5*\cell) {$#7$};   
    \node at (0.5*\cell,-4.5*\cell) {$#8$};   

    #9
\end{scope}
}

\def\dx{2.8}

\gtabloid{0}{0}{a}{c}{d}{f}{b}{e}{
    \draw[thick] (0.5*\cell,-3.2*\cell) -- (0.5*\cell,-2.8*\cell);
    \draw[thick] (0.8*\cell,-0.5*\cell) -- (1.2*\cell,-0.5*\cell);
}

\gtabloid{\dx}{0}{b}{f}{d}{c}{a}{e}{
    \draw[thick] (0.5*\cell,-3.2*\cell) -- (0.5*\cell,-2.8*\cell);
    \draw[thick] (0.8*\cell,-0.5*\cell) -- (1.2*\cell,-0.5*\cell);
}

\gtabloid{2*\dx}{0}{b}{e}{a}{d}{f}{c}{
    \draw[thick] (0.5*\cell,-0.8*\cell) -- (0.5*\cell,-1.2*\cell);
    \draw[thick] (0.8*\cell,-0.5*\cell) -- (1.2*\cell,-0.5*\cell);
}


\gtabloid{3*\dx}{0}{a}{c}{f}{b}{e}{d}{
    \draw[thick] (0.5*\cell,-1.8*\cell) -- (0.5*\cell,-2.2*\cell);
    \draw[thick] (0.8*\cell,-0.5*\cell) -- (1.2*\cell,-0.5*\cell);
}

\end{tikzpicture}
\end{figure}

\end{example}

By considering SRH $G$-tabloids, we may obtain an alternative combinatorial interpretation of Schur coefficients, as is done in \cite[Corollary 3.1]{gen_nets}.

\begin{proposition}[\cite{gen_nets}]
\label{prop:better_s_form}
Consider any graph $G$, a partition $\lambda$ of $|V(G)|$, and a partial order on the vertices of $G$ such that non-adjacent vertices are comparable. We have
\[
[s_{\lambda}]X_G=\sum_{T\in \mathcal{T}_{\lambda,G}}\mathrm{sgn}(T).
\]
\end{proposition}



 \section{Complete Multipartite Graphs}
\label{sect:comp_multi_graphs}

 
 \begin{definition}
Let $\lambda=(\lambda_1,\ldots,\lambda_k)$ be a partition. A \textit{complete multipartite graph} $K_{\lambda}$ is a graph with $k$ stable sets of vertices of respective sizes $\lambda_1,\ldots,\lambda_k$ such that every possible edge between stable sets is included. 

\end{definition}

\begin{example} The complete multipartite graph $K_{(4,2,2)}$ is shown below.
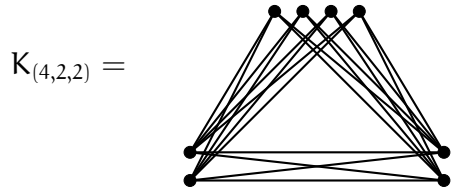
\label{ex:comp_multi}

\begin{figure}[H]
\centering
\begin{tikzpicture}
\node[] at (-2,1.5) {$K_{(4,2,2)}=$};
\draw[fill=black,scale=.75] (-.5,0) circle (3pt);
\draw[fill=black,scale=.75] (-.5,.5) circle (3pt);
\draw[fill=black,scale=.75] (4,0) circle (3pt);
\draw[fill=black,scale=.75] (4,.5) circle (3pt);
\draw[fill=black,scale=.75] (1.5,3) circle (3pt);
\draw[fill=black,scale=.75] (1,3) circle (3pt);
\draw[fill=black,scale=.75] (2,3) circle (3pt);
\draw[fill=black,scale=.75
] (2.5,3) circle (3pt);
\draw[thick,scale=.75] (1,3) -- (4,.5);
\draw[thick,scale=.75] (1,3) -- (4,0);
\draw[thick,scale=.75] (1,3) -- (-.5,0);
\draw[thick,scale=.75] (1,3) -- (-.5,.5);
\draw[thick,scale=.75] (1.5,3) -- (4,.5);
\draw[thick,scale=.75] (1.5,3) -- (4,0);
\draw[thick,scale=.75] (1.5,3) -- (-.5,0);
\draw[thick,scale=.75] (1.5,3) -- (-.5,.5);
\draw[thick,scale=.75] (2,3) -- (4,.5);
\draw[thick,scale=.75] (2,3) -- (4,0);
\draw[thick,scale=.75] (2,3) -- (-.5,0);
\draw[thick,scale=.75] (2,3) -- (-.5,.5);
\draw[thick,scale=.75] (2.5,3) -- (4,.5);
\draw[thick,scale=.75] (2.5,3) -- (4,0);
\draw[thick,scale=.75] (2.5,3) -- (-.5,0);
\draw[thick,scale=.75] (2.5,3) -- (-.5,.5);
\draw[thick,scale=.75] (-.5,0) -- (4,.5);
\draw[thick,scale=.75] (-.5,0) -- (4,0);
\draw[thick,scale=.75] (-.5,.5) -- (4,.5);
\draw[thick,scale=.75] (-.5,.5) -- (4,0);
\end{tikzpicture}
\end{figure}
\end{example}

All complete multipartite graphs are incomparability graphs because $K_{\lambda}$ is precisely the incomparability graph of the poset consisting of disjoint chains of lengths $\lambda_1,\ldots,\lambda_k$. \emph{For the remainder of this paper, we will hence view $K_\lambda$ as the incomparability graph of this poset, where its partial order is inherited from the poset.} As can be observed in Example \ref{ex:comp_multi}, any complete multipartite graph $K_{\lambda}$ with $\lambda_1\ge 3$ and $k\ge 2$ contains an induced copy of the claw graph. Therefore, most complete multipartite graphs are neither covered by the resolved Stanley-Stembridge Conjecture \cite{hikita24} nor the open Claw-Free Conjecture \cite{stan_98}. In this section, we will show that these graphs are not Schur-positive for most choices of $\lambda$.

In \cite{WW2020}, a Schur-positivity classification is proved for complete bipartite and complete tripartite graphs. Specifically, \cite[Example 2.5 and Theorem 3.2]{WW2020} determine that the only Schur-positive complete bipartite graphs are
\[
K_{\lambda}\qquad \text{for}\qquad \lambda\in \{(1,1), (2,1), (2,2), (3,2)\},
\] while the only Schur-positive complete tripartite graphs are
\[
K_{\lambda}\qquad \text{for}\qquad\lambda\in\{(1,1,1), (2,1,1), (2,2,1), (2,2,2), (3,2,2)\}.
\]
We extend this result to classify all complete multipartite graphs in terms of Schur-positivity in Theorem \ref{thm:classif}. We do not consider complete multipartite graphs $K_{\lambda}$ with $\ell(\lambda)=1$ since these are trivially $\mathbf{e}$-positive as the union of $P_1$ graphs. We begin with the following definition, which characterizes a special type of partition.


\begin{definition}
A partition $\lambda=(\lambda_1,\ldots,\lambda_k)$ is \textit{balanced} if $\lambda_1\le \lambda_k+1$. 
\end{definition}

For example, the partitions $(5,5,5,4)$ and $(2,1,1,1)$ are balanced whereas $(3,2,1)$ and $(5,5,2,2)$ are not. The next proposition is a useful tool for showing graphs are not Schur-positive. 

\begin{proposition}
\label{prop:only_balanced_partitions}
Assume $G$ is Schur-positive and has a unique stable $k$-partition, which is of type $\lambda=(\lambda_1,\ldots,\lambda_k)$. Then, $\lambda$ is balanced.
\end{proposition}
\begin{proof}

Since $G$ is Schur-positive with a unique stable $k$-partition $\lambda$, there is no partition $\mu$ of length $k$ which is dominated by $\lambda$. Otherwise, by Proposition \ref{prop:schur_pos_means_nice}, $G$ would have a stable $k$-partition of type $\mu$ which is a contradiction to the uniqueness of $\lambda$.



Assume $\lambda$ is not balanced so $\lambda_1>\lambda_k+1$. Choose the largest $1\le j \le k-1$ such that $\lambda_j=\lambda_1$. Likewise, choose the smallest $2\le i \le k$ such that $\lambda_i=\lambda_k$. We then have
\begin{align*}
\lambda&=(\lambda_1,\ldots,\lambda_j,\lambda_{j+1},\ldots, \lambda_{i-1},\lambda_i,\ldots,\lambda_k)\\
&\ge (\lambda_1,\ldots,\lambda_j-1,\lambda_{j+1},\ldots,\lambda_{i-1},\lambda_{i}+1,\ldots,\lambda_k),
\end{align*}
which is a contradiction to the first sentence of the proof. Note that the latter is a valid partition since
\[
\lambda_1-1=\lambda_j-1 \ge \lambda_{j+1}\ge\cdots\ge \lambda_{i-1}\ge {\lambda_i+1}=\lambda_k+1.
\]
We conclude that $\lambda$ must be balanced.
\end{proof}



We can now show that only a special subclass of complete multipartite graphs have the potential to be Schur-positive.

\begin{corollary}
\label{cor:only_balanced_comp_graphs}
All complete multipartite graphs $K_{\lambda}$ for which $\lambda$ is not balanced are not Schur-positive and therefore not $\mathbf{e}$-positive. 
\end{corollary}
\begin{proof}
Complete multipartite graphs $K_{\lambda}$ with $\ell (\lambda)=k$ have a unique stable $k$-partition of type $\lambda$, obtained by dividing the vertices into the $k$ stable $\lambda_i$-sets given by the definition of $K_\lambda$. This is because if any of these $\lambda_i$-sets is separated into two or more stable sets, there are $k-1$ stable $\lambda_i$-sets leftover to divide into at most $k-2$ stable sets. This is impossible since $K_{\lambda}$ includes  all possible edges between the stable $\lambda_i$-sets. Hence, if $K_{\lambda}$ is Schur-positive, then $\lambda$ is balanced by Proposition \ref{prop:only_balanced_partitions}, and we are done.
\end{proof}

To illustrate the last two results, we consider an example of a complete multipartite graph which is not Schur-positive.

\begin{example}
For the sake of contradiction, assume that the graph $G=K_{(5,5,5,4,3,3)}$ is Schur-positive. We note that the partition $(5,5,5,4,3,3)$ is not balanced. We then have
\[
(5,5,5,4,3,3)\ge (5,5,4,4,4,3).
\]
Since $G$ has a stable partition of type $(5,5,5,4,3,3)$, $G$  also has a stable partition of type $(5,5,4,4,4,3)$ by Proposition \ref{prop:schur_pos_means_nice}. However, no such stable partition exists so we have a contradiction, and $K_{(5,5,5,4,3,3)}$ is not Schur-positive.
\end{example}

Now that we have reduced the cases to consider to only complete multipartite graphs $K_{\lambda}$ where $\lambda$ is balanced, we use a similar argument to reduce the cases even further.

\begin{theorem}
\label{thm:only_K32_may_b_s_pos}
If $K_{\lambda}$ is a complete multipartite graph with $\ell(\lambda)\ge 2$, then $K_{\lambda}$ is $\mathbf{e}$-positive, and hence Schur-positive, if $\lambda_i\in\{1,2\}$. Otherwise, if $\lambda$ is not $(3,2^{\beta})$ for  some $\beta\ge 1$, then $K_{\lambda}$ is not Schur-positive. 
\end{theorem}
\begin{proof}
By Corollary \ref{cor:only_balanced_comp_graphs}, it suffices to consider only  the graphs $K_{\lambda}$ for which $\lambda$ is balanced.
\begin{enumerate}
    \item Firstly, assume $\lambda_i\in\{1,2\}$ for all $i$. We have that $K_{\lambda}$ is $K_{(3,1)}$-free since $K_{\lambda}$ has no stable $3$-set. Thus, $K_{\lambda}$ is $\mathbf{e}$-positive because it is a claw-free incomparability graph \cite{hikita24}, and hence Schur-positive.

    
    
    \item Secondly, assume $\lambda=(m^{\alpha},(m-1)^{\beta})$ for $\alpha\ge 2$, $\beta\ge 0$, $m\ge 3$. We then have
    \[
    (m^{\alpha},(m-1)^{\beta})\ge (m^{\alpha-2},(m-1)^{\beta+2},2).
    \]
However, $K_{\lambda}$ does not have a stable partition of type $(m^{\alpha-2},(m-1)^{\beta+2},2)$. To see this, assume otherwise. Then the $\alpha-2$ $m$-parts must be $\alpha -2$ of the $\alpha$ stable $m$-sets. 
We then have two stable $m$-sets leftover, which must each be split into different parts in the stable partition. If one of them yields an $(m-1)$-part, there is only one vertex leftover, which cannot be part of any other stable set of the desired size as it is adjacent to all other vertices. If one of them yields the $2$-part, the remaining $m-2$ vertices cannot be part of any other stable set of the desired size as they are adjacent to all other vertices. Hence, $K_{\lambda}$ is not Schur-positive by Proposition \ref{prop:schur_pos_means_nice}.

\item Lastly, assume $\lambda=(m,(m-1)^{\beta})$ for $\beta \ge 2, m\ge 4$. We then have
\[
(m,(m-1)^{\beta})\ge (m,(m-1)^{\beta-2},(m-2)^2,2).
\]
However, $K_{\lambda}$ does not have a stable partition of type $(m,(m-1)^{\beta-2},(m-2)^2,2)$. To see this, assume otherwise. Then the $m$-part must be the unique stable $m$-set and the $\beta-2$ $(m-1)$-parts must be $\beta -2$ of the $\beta$ stable $(m-1)$-sets. We then have two $(m-1)$-sets leftover, which must each be split into different parts in the stable partition. If one of them yields an $(m-2)$-part, there is only one vertex leftover, which cannot be part of any other stable set of the desired size as it is adjacent to all other vertices. If one of them yields the $2$-part, the remaining $m-3$ vertices cannot be part of any other stable set of the desired size as they are adjacent to all other vertices. Hence, $K_{\lambda}$ is not Schur-positive by Proposition \ref{prop:schur_pos_means_nice}.
\end{enumerate}

The case where $\lambda=(m,m-1)$ is covered by the classification of complete bipartite graphs in \cite[Example 2.5]{WW2020}. The only general case which remains is when $\lambda=(3,2^{\beta})$, $\beta\ge 2$, and we are done.
\end{proof}

To illustrate the last result, we provide two examples of complete multipartite graphs which are not Schur-positive. 
\begin{example}
For sake of contradiction, assume the graph $G=K_{(6,6,5,5,5)}$ is Schur-positive. This graph falls under the second case in the prior proof. We have
\[
(6,6,5,5,5)\ge (5,5,5,5,5,2).
\]
Since $G$ has a stable partition of type $(6,6,5,5,5)$, $G$ also has a stable partition of type $(5,5,5,5,5,2)$ by Proposition \ref{prop:schur_pos_means_nice}. However, no such stable partition exists.

Next, for sake of contradiction, assume the graph $G=K_{(5,4,4,4)}$ is Schur-positive.
This graph falls under the third case in the prior proof. We have
\[
(5,4,4,4)\ge (5,4,3,3,2).
\]
Since $G$ has a stable partition of type $(5,4,4,4)$, $G$  also has a stable partition of type $(5,4,3,3,2)$ by Proposition \ref{prop:schur_pos_means_nice}. However, no such stable partition exists.

In each case we have a contradiction, and hence neither $K_{(6,6,5,5,5)}$ nor $K_{(5,4,4,4)}$ are Schur-positive.

\end{example}



We next prove Theorem \ref{thm:inv_paths_tail}, which significantly reduces the number of SRH $G$-tabloids which must be counted to compute the Schur coefficients of the chromatic symmetric function of an incomparability graph of $G$. For this we need two more definitions.

\begin{definition}
Let $T$ be an SRH $G$-tabloid. We define the \textit{tail sequence of $T$}, denoted $\text{ts}(T)$, to be the sequence of vertices of $G$ in the rows of length $1$ of the tabloid $T$, read from bottom to top.
\end{definition}

\begin{definition}
Let $(P,\le)$ be a finite poset and $G=\text{inc}(P)$. A \textit{non-increasing sequence in $G$} is a sequence $(v_1,\ldots,v_k)$ of vertices of $G$ such that $v_i \nleq v_{i+1}$ for $1\le i \le k-1$. A \textit{spanning non-increasing sequence of $G$} is a non-increasing sequence in $G$ which includes every vertex of $G$ exactly once. Let $N_{\mathrm{sp}}(G)$ denote the number of spanning non-increasing sequences of $G$. We note that if $G$ has no vertices, then $N_{\mathrm{sp}}(G)=1$, representing the empty non-increasing sequence in $G$. 
\end{definition}

To illustrate the prior two definitions, we note that the tail sequences of the SRH $G$-tabloids in Example \ref{ex:SRH_G_tab} are respectively
\[
(e,b,f,d),\qquad (e,a,c,d),\qquad (c,f,d,a),\andd (d,e,b,f). 
\]
Among these, only $(c,f,d,a)$ is a non-increasing sequence in the underlying incomparability graph \steph{that was originally defined in Examples \ref{ex:hasse} and \ref{ex:inc_graph}.} An example of a spanning non-increasing sequence of this incomparability graph is the sequence $(f,e,b,c,a,d)$.

To make way for the following theorem, we now set the notation
\[
\tilde{\mathcal{T}}_{\lambda,G}=\{ T\in \mathcal{T}_{\lambda,G} \text{ such that } \text{ts}(T)\text{ is a non-increasing sequence in $G$}\},
\]
which may be used when $G$ is an incomparability graph.
\begin{theorem}
\label{thm:inv_paths_tail}
For any incomparability graph $G$, we have
\[
[s_{\lambda}]X_G=\sum_{T\in\tilde{\mathcal{T}}_{\lambda,G}}\mathrm{sgn}(T).
\]
\end{theorem}
\begin{proof}
Consider any incomparability graph $G$ and partition $\lambda$. From Proposition \ref{prop:better_s_form}, we have
\[
[s_{\lambda}]X_G=\sum_{T\in\mathcal{T}_{\lambda,G}}\mathrm{sgn}(T).
\]
Suppose $\lambda$ has $k$ rows of length $1$. We denote by $S_k(G)$ the set of all sequences of distinct vertices of $G$ of length $k$. Moreover, we set
\[
\mathcal{T}_{\lambda,G,s}=\{T\in\mathcal{T}_{\lambda,G}\text{ such that }\text{ts}(T)=s\}.
\]
Hence, we have
\[
[s_{\lambda}]X_G=\sum_{s\in S_k(G)}\sum_{T\in\mathcal{T}_{\lambda,G,s}}\mathrm{sgn}(T).
\]
For any $s\in S_k(G)$ and $T\in \mathcal{T}_{\lambda,G,s}$, we set $j$ to be the minimal index $1\le j \le k-1$ such that $v_j\le v_{j+1}$ (if it exists). We let $e_j$ denote an $N$-step from $v_j$ to $v_{j+1}$.

Now, we define the map
\begin{align*}
\psi: \mathcal{T}_{\lambda,G,s} &\to \mathcal{T}_{\lambda,G,s}\\
T &\mapsto T\Delta e_j,
\end{align*}
where $T\Delta e_j$ is the SRH $G$-tabloid obtained by removing $e_j$ if $e_j$ is already included in $T$ and adding $e_j$ otherwise.

If $v_{j+1}$ is part of a rim hook, it must be of the form \steph{$v_{j+1}\le v_{j+2}\le\cdots\le v_{j+h}$} (read bottom to top). Thus, by transitivity of the underlying poset, \steph{$v_{j}\le v_{j+1}\le v_{j+2}\le\cdots\le v_{j+h}$} so $v_j$ may be added to the rim hook. Note that there are no $N$-steps below $v_j$ since that would contradict the minimality of $j$.

Since the sign of an SRH $G$-tabloid is given by the parity of the \steph{number of} $N$-steps, we have that $\psi$ is a sign-reversing involution for any set $\mathcal{T}_{\lambda,G,s}$ on which it is defined. Therefore, whenever there exists a minimal $j$ such that $v_j\le v_{j+1}$ in the sequence $s$, we have
\[
\sum_{T\in\mathcal{T}_{\lambda,G,s}}\mathrm{sgn}(T)=0.
\]
We are thus left with the sums ranging over $\mathcal{T}_{\lambda,G,s}$ for which there is no index $j$ such that $v_j\le v_{j+1}$ in the sequence $s$. These sequences are precisely the non-increasing sequences of length $k$ in $G$. We conclude the desired formula holds. 
\end{proof}

Now \steph{we} apply Theorem \ref{thm:inv_paths_tail} to derive \steph{an explicit combinatorial} non-negative formula for the Schur coefficients of the chromatic symmetric functions of complete multipartite graphs $K_{(2^{\beta})}$, $\beta\ge 1$. 
We note that this formula covers all Schur coefficients for the graphs $K_{(2^{\beta})}$. \steph{This is because} the largest stable set in $K_{(2^{\beta})}$ has $2$ vertices, \steph{so} it is impossible for an SRH $K_{(2^{\beta})}$-tabloid to have a rim hook of length $3$ or greater and therefore impossible to have a row of length $3$ or greater.


\begin{lemma}
\label{lemma:form_two_partite_case}
We have \steph{for $\beta \geq 1$}
\[
[s_{(2^C,1^D)}]X_{K_{(2^{\beta})}}=\frac{\beta !}{(\beta-C)!}\cdot N_{\mathrm{sp}}(K_{(2^{\beta-C})})
\]
for $C,D\ge 0$ such that $2C+D=2\beta$.
\end{lemma}
\begin{proof}
Set $G=K_{(2^{\beta})}$. For any non-zero coefficient $[s_{(2^C,1^D)}]X_{G}$, we must have $2C+D=2\beta$ in order for the number of cells in the tabloid to match the number of vertices in $G$. For any $T\in \tilde{\mathcal{T}}_{(2^C,1^D),G}$, the head of the tabloid $T$ must only consist of $E$-hooks. This is because the largest stable set in $G$ contains two vertices, so there are no rim hooks of length $3$ in the diagram. In particular, this also implies there cannot be an $N$-step up from the top row of $\mathrm{tl}(T)$ to the bottom row of $\mathrm{hd}(T)$.

In order to fill $\mathrm{hd}(T)$, we must select $C$ stable $2$-sets from $G$ and order them. Since $G$ is made up of $\beta$ distinct stable $2$-sets to choose from, we have $\frac{\beta!}{(\beta-C)!}$ choices for this part of the diagram. In order to fill $\mathrm{tl}(T)$, we must choose a non-increasing sequence in $G$ with the remaining vertices of $G$ \steph{by Theorem~\ref{thm:inv_paths_tail}.} These vertices form the subgraph $K_{(2^{\beta-C})}$ so this yields $N_{\mathrm{sp}}(K_{(2^{\beta-C})})$ choices.

Next, we note there are no $N$-steps in $\mathrm{hd}(T)$. Likewise, there are no $N$-steps in $\mathrm{tl}(T)$ since the vertices in the tail form a non-increasing sequence in $G$. Accordingly, all SRH $G$-tabloids we have counted have positive sign. We conclude the desired formula holds.
\end{proof}

We  \steph{now} employ the formula in Lemma \ref{lemma:form_two_partite_case} to assist in explicitly characterizing all the Schur coefficients of the chromatic symmetric function of $K_{(3,2^{\beta})}$ in terms of non-increasing path counts. This will complete the Schur-positivity 
classification of complete multipartite graphs.

\begin{theorem}
\label{thm:K32_s_pos}
 We have that $K_{(3,2^{\beta})}$ is Schur-positive for $\beta\ge 1$. Moreover, we have
\begin{equation}
\label{eq:target_coeff_first}
[s_{(3,2^C,1^D)}]X_{K_{(3,2^{\beta})}}=[s_{(2^C,1^D)}]X_{K_{(2^{\beta})}}=\frac{\beta !}{(\beta-C)!}\cdot N_{\mathrm{sp}}(K_{(2^{\beta-C})})
\end{equation}
for $C,D\ge 0$ such that $2C+D=2\beta$. We also have
\begin{align}
\begin{split}
\label{eq:target_coeff_second}
[s_{(2^C,1^D)}]X_{K_{(3,2^{\beta})}}&=\frac{\beta !}{(\beta - C+1)!} \Big((\beta-C+1)\cdot N_{\mathrm{sp}}(K_{(3,2^{\beta-C})})-N_{\mathrm{sp}}(K_{(2^{\beta-C+1})})\\
&+(C+2)\cdot N_{\mathrm{sp}}(K_{(2^{\beta-C+1},1)})\Big)
\end{split}
\end{align}
for $C\ge 1,D\ge 1$ such that $2C+D=2\beta+3$. We note that if $\beta<C$, $N_{\mathrm{sp}}(K_{(3,2^{\beta-C})})$ is set to equal zero since the graph $K_{(3,2^{\beta-C})}$ does not exist. Lastly, we have
\begin{equation}
\label{eq:target_coeff_third}
[s_{1^{2\beta+3}}]X_{K_{(3,2^{\beta})}}=N_{\mathrm{sp}}(K_{(3,2^{\beta})}).
\end{equation}
\end{theorem}
\begin{proof}
Fix $G=K_{(3,2^{\beta})}$, $\beta\ge 1$. Denote the three vertices in the unique stable $3$-set by $v_1< v_2< v_3$, where the order is given by the underlying poset. We divide the coefficients $
[s_{\lambda}]X_{G}$ into several cases. \steph{Using Theorem \ref{thm:inv_paths_tail}, we need only consider tabloids in $\tilde{\mathcal{T}}_{\lambda,G}$ throughout.}
\begin{enumerate}
    \item Assume $\lambda=(3,2^C,1^D)$ for $C,D\ge 0$. We must have $2C+D=2\beta$ in order for the number of cells in the tabloids to match the number of vertices in $G$. For any $T\in \tilde{\mathcal{T}}_{\lambda,G}$, the row of length $3$ must be filled by the unique stable $3$-set in $G$. The remainder of the diagram must then be filled with the rest of the vertices, which form a subgraph $K_{(2^{\beta})}$. Accordingly, these tabloids are in sign-preserving bijection with the tabloids in $\tilde{\mathcal{T}}_{(2^{C},1^D),K_{(2^{\beta})}}$. Therefore, Equation \ref{eq:target_coeff_first} follows from Lemma \ref{lemma:form_two_partite_case}.

    \item Assume $\lambda=(2^C,1^D)$ for $C\ge 1$, $D\ge 1$. We must have $2C+D=2\beta+3$ in order for the number of cells in the tabloids to match the number of vertices in $G$. We also observe that $D\ge 1$ must hold since $G$ has an odd number of vertices. We consider several subcases for $T\in\tilde{\mathcal{T}}_{\lambda,G}$. 
    \begin{enumerate}[(a)]
    \item First, $T$ \steph{has} no $N$-steps. In this case, the rows of length $2$ are all filled with $E$-hooks. We have that $v_1,v_2,v_3$ cannot all be in the head since they are adjacent to all other vertices and thus \steph{cannot} be placed in two stable $2$-sets. It is also impossible for there to be exactly two of the three vertices $v_1,v_2,v_3$ in the tail since that would leave one of the vertices in the head with no other vertex to form a stable $2$-set. Thus, we break into two subcases.
    \begin{enumerate}[(i)]
        \item Assume $v_1,v_2,v_3$ are all in the tail. We choose $C$ ordered stable $2$-sets from the $\beta$ stable $2$-sets in $G$ to fill the head. Then, we choose a non-increasing sequence in $G$ through all the remaining vertices to fill the tail. This yields
        \[
        \frac{\beta !}{(\beta-C)!}\cdot N_{\mathrm{sp}}(K_{(3,2^{\beta-C})})
        \]
        positive SRH $G$-tabloids.

        
        
        \item Assume one vertex $v_j$ from the stable $3$-set is in the tail. We have three choices for the vertex $v_j$. We then choose $C-1$ ordered stable $2$-sets out of the $\beta$ total stable $2$-sets. Next, we choose a position for the $E$-hook with the other two vertices from the $3$-set, yielding $C$ possible choices. Finally, we choose a non-increasing sequence in $G$ through the remaining vertices. We thus count
        \[
        3\cdot \frac{\beta!}{(\beta-C+1)!}\cdot C \cdot N_{\mathrm{sp}}(K_{(2^{\beta-C+1},1)})
        \]
        positive $SRH$ $G$-tabloids. 
    \end{enumerate}
        \item Second, $T$ \steph{has} an $NE$-hook containing $v_1<v_2<v_3$ and starting from the top cell in the tail. The $N$-step in this rim hook must be the only $N$-step in the diagram since the rest of the head is filled with \steph{$E$-hooks} and the vertices in the tail compose a non-increasing sequence in $G$. We choose $C-1$ ordered stable $2$-sets out of the $\beta$ total stable $2$-sets to fill the rest of the head. Then, we choose a non-increasing sequence through the remaining vertices to fill the rest of the tail. We note that a non-increasing sequence followed by $v_1$ must still be a non-increasing sequence since $v_1$ is incomparable to all vertices besides $v_2$ and $v_3$. Hence, we count
        \[
        \frac{\beta !}{(\beta-C+1)!}\cdot N_{\mathrm{sp}}(K_{(2^{\beta-C+1})})
        \]
        negative SRH $G$-tabloids.
        \item Third, $T$ \steph{has} an $EN$-hook containing $v_1<v_2<v_3$ and located somewhere in the head. Once again, the $N$-step in this rim hook must be the only $N$-step in the diagram since the rest of the head is filled with $E$-hooks and the vertices in the tail compose a non-increasing sequence in $G$. There are $C-1$ positions for the $EN$-hook. We choose $C-2$ ordered stable $2$-sets from the $\beta$ total stable $2$-sets to fill the other rows in the head. Moreover, there must be a rim hook of length $1$ to the northwest of the \steph{$EN$-hook} to complete the diagram. There remain $2(\beta-C+2)$ vertices from which to choose this rim hook. Finally, we choose a non-increasing sequence through the remaining vertices. This yields
        \[
        (C-1)\cdot \frac{\beta !}{(\beta-C+2)!}\cdot 2\cdot (\beta-C+2)\cdot N_{\mathrm{sp}}(K_{(2^{\beta-C+1},1)})
        \]
        negative SRH $G$-tabloids. We note that this case only occurs if $C\ge 2$. Otherwise, the above term vanishes.
    \end{enumerate}
    To derive the desired formula, we now compute the sum of all the above terms \steph{for this case.}
    \begin{align*}
        [s_{\lambda}]X_{G}&=\frac{\beta !}{(\beta-C)!}\cdot N_{\mathrm{sp}}(K_{(3,2^{\beta-C})})+3\cdot \frac{\beta!}{(\beta-C+1)!}\cdot C \cdot N_{\mathrm{sp}}(K_{(2^{\beta-C+1},1)})\\
        &-\frac{\beta !}{(\beta-C+1)!}\cdot N_{\mathrm{sp}}(K_{(2^{\beta-C+1})})-(C-1)\cdot \frac{\beta !}{(\beta-C+1)!}\cdot 2\cdot N_{\mathrm{sp}}(K_{(2^{\beta-C+1},1)})\\
        &=\frac{\beta!}{(\beta-C+1)!}\Big((\beta-C+1)\cdot N_{\mathrm{sp}}(K_{(3,2^{\beta-C})})+3C\cdot N_{\mathrm{sp}}(K_{(2^{\beta-C+1},1)})\\
        &-N_{\mathrm{sp}}(K_{(2^{\beta-C+1})})-(2C-2)\cdot N_{\mathrm{sp}}(K_{(2^{\beta-C+1},1)})\Big)\\
        &=\frac{\beta !}{(\beta - C+1)!} \Big((\beta-C+1)\cdot N_{\mathrm{sp}}(K_{(3,2^{\beta-C})})-N_{\mathrm{sp}}(K_{(2^{\beta-C+1})})\\
        &+(C+2)\cdot N_{\mathrm{sp}}(K_{(2^{\beta-C+1},1)})\Big)
    \end{align*}
    Since this formula has a negative term, we must prove it is always non-negative in order to ensure Schur-positivity. To do so, we prove the inequality
    \begin{equation}
    \label{eq:nsp_ineq}
    N_{\mathrm{sp}}(K_{(2^{\beta-C+1},1)})\ge N_{\mathrm{sp}}(K_{(2^{\beta-C+1})}).
    \end{equation}
    We have that $K_{(2^{\beta-C+1},1)}$ is the incomparability graph of the poset consisting of $\beta-C+1$ disjoint chains of length $2$ and $1$ disjoint chain of length $1$. We label the vertex corresponding to the chain of length $1$ by $u$.
    
    Consider any non-increasing sequence $v_1\nleq\cdots\nleq v_{2(\beta-C+1)}$ through all the vertices of $K_{2^{\beta-C+1}}$. We can map any such sequence to the non-increasing sequence $v_1\nleq\cdots\nleq v_{2(\beta-C+1)}\nleq u$ since $u$ is incomparable to all other vertices. This injection implies the inequality in Equation \ref{eq:nsp_ineq}.

    \item Lastly, if $\lambda=(1^{2\beta+3})$, the SRH $G$-tabloids in $\tilde{\mathcal{T}}_{\lambda,G}$ are in bijection with the spanning non-increasing sequences of $G$ by Theorem \ref{thm:inv_paths_tail}. These tabloids are all positive with no $N$-steps, so Equation \ref{eq:target_coeff_third} follows.

\end{enumerate}
Since the largest stable set in $G$ is the unique stable $3$-set, the SRH $G$-tabloids we count \steph{can} contain at most one rim hook of length $3$ and no longer rim hooks. Therefore, these tabloids \steph{can} have at most one row of length $3$ and no longer rows. Accordingly, Equations \ref{eq:target_coeff_first}-\ref{eq:target_coeff_third} cover all possible partitions corresponding to non-zero coefficients. We conclude $G=K_{(3,2^{\beta})}$ is Schur-positive.
\end{proof}

We note that it is possible to derive explicit combinatorial formulas for counts of non-increasing spanning paths such as $N_{\mathrm{sp}}(K_{(2^{\beta})})$. However, since these counts are inherently non-negative, our argument for the Schur-positivity of $K_{(3,2^{\beta})}$ does not require such formulas.

Together with Theorem \ref{thm:only_K32_may_b_s_pos}, Theorem \ref{thm:K32_s_pos} provides the following Schur-positivity classification for complete multipartite graphs, \steph{and our main result.}

\begin{theorem}
\label{thm:classif}
A complete multipartite graph $K_{\lambda}$ with $\ell(\lambda)\ge 2$ is Schur-positive if and only if $\lambda_i\in\{1,2\}$ for $1\le i \le \ell(\lambda)$ or $\lambda=(3,2^{\beta})$ for some $\beta\ge 1$.
\end{theorem}

Theorem \ref{thm:only_K32_may_b_s_pos} also classifies most complete multipartite graphs in terms of $\mathbf{e}$-positivity \steph{since} graphs which are not Schur-positive are not $\mathbf{e}$-positive. The question of whether $K_{(3,2^{\beta})}$ ($\beta\ge 1$) is $\mathbf{e}$-positive, however, remains open. Lastly, we remark that the \steph{SRH $G$-tabloid} counting methods used throughout this paper---and, in particular, Theorem \ref{thm:inv_paths_tail}---have the potential to serve as tools for characterizing the Schur-positivity of other incomparability graphs which are not claw-free.

\bibliographystyle{siam}
\bibliography{biblio}

\end{document}